\documentclass[12pt]{article}
\usepackage{amsmath,amssymb,amstext,dsfont,fancyvrb,float,fontenc,graphicx,subfigure,theorem}
\usepackage[unicode]{hyperref}

\usepackage[letterpaper]{geometry}
\ifx\volno\undefined\def\volno{0}\fi
\ifx\volyear\undefined\def\volyear{2017}\fi
\ifx\pagno\undefined\def\pagno{000--000}\fi

\newfont{\footsc}{cmcsc10 at 8truept}
\newfont{\footbf}{cmbx10 at 8truept}
\newfont{\footrm}{cmr10 at 10truept}

\usepackage{fancyhdr}
\usepackage{relsize}
\usepackage{sectsty}
\allsectionsfont{\larger[-1]} 

\renewcommand\paragraph{\@startsection{paragraph}{4}{\z@}
                                    {2ex \@plus.5ex \@minus.2ex}
                                    {-1em}
                                    {\normalfont\normalsize\bfseries}}

\renewcommand\subparagraph{\@startsection{subparagraph}{5}{\parindent}
                                       {2ex \@plus.5ex \@minus .2ex}
                                       {-1em}
                                      {\normalfont\normalsize\bfseries}}

\newlength{\BiblioSpacing}
\renewenvironment{thebibliography}[1]{
\begin{oldthebibliography}{#1}
\setlength{\parskip}{\BiblioSpacing}
\setlength{\itemsep}{\BiblioSpacing}
}
{
\end{oldthebibliography}
}

\usepackage[strict]{changepage}
\def\abstractname{Abstract -}   % <-----------------
\def\abstract{\begin{adjustwidth}{1cm}{1cm} \par    \footnotesize \noindent {\bf \abstractname} 
\def\endabstract{ \end{adjustwidth} \smallskip }}
{\theorembodyfont{\itshape}\newtheorem{theorem}{Theorem}[section]}
{\theorembodyfont{\itshape}\newtheorem{proposition}[theorem]{Proposition}}
{\theorembodyfont{\itshape}\newtheorem{definition}[theorem]{Definition}}
{\theorembodyfont{\itshape}}
{\theorembodyfont{\itshape}\newtheorem{corollary}[theorem]{Corollary}}
{\theorembodyfont{\rm}}
{\theorembodyfont{\rm}}
{\theorembodyfont{\rm}}
{\theorembodyfont{\rm }\newtheorem{example}[theorem]{Example}}

\def\dedicatory{\date}

\usepackage{mathtools}
\usepackage{enumerate}
\usepackage[all]{xy}
\usepackage{pgf, tikz}
\usetikzlibrary{shapes,decorations,decorations.markings,arrows,calc,arrows.meta,fit,positioning}
\usepackage{verbatim}
\usepackage{caption}
\usepackage{float}

\DeclarePairedDelimiter{\ceil}{\lceil}{\rceil}

\newcommand{\Zz}{\mathbb{Z}}
\newcommand{\Rr}{\mathbb{R}}
\newcommand{\Nn}{\mathbb{N}}

\renewcommand{\vec}{\mathbf}

\newcommand{\vx}{\vec{x}}

\newcommand{\vd}{\vec{d}}
\newcommand{\vp}{\vec{p}}
\newcommand{\vq}{\vec{q}}
\newcommand{\vr}{\vec{r}}
\newcommand{\cz}{ {\color{red} 0} } %color zero
\newcommand{\cy}{ {\color{blue} 1} } %color one

\newcommand{\pEN}{\mathrm{pEN}}
\newcommand{\sEN}{\mathrm{sEN}}

\newcommand{\phd}{\underline{\hspace{2ex}}}
\renewcommand{\le}{\leqslant}
\renewcommand{\ge}{\geqslant}

\title{\Large\bf Refinement of Metrics: Erd\H{o}s Number, a Case Study}
  \author{\sc K. Lock, W.Y. Pong, and A. Wittmond\footnote{This work was
  supported by a PUMP research grant (NSF grant No. DMS-1247679). The first
  and the third author were undergraduate students of CSU Dominguez-Hills.}}
\dedicatory{\normalsize\em}

\begin{document}
\setcounter{page}{1}
%\date{}
\maketitle
\thispagestyle{fancy}

\vskip 1.5em

\begin{abstract}
  We introduce a concept called refinement and develop two different ways of
  refining metrics. By applying these methods we produce several refinements
  of the shortest-path distance on the collaboration graph and hence a couple
  new versions of the Erd\H{o}s number.
\end{abstract}
 

\begin{MSC}
05C12; 05C90
\end{MSC}

\section{Introduction} \label{sec:mot}

Our investigation was motivated by a simple goal: to find a ``better''
version of the Erd\H{o}s number.

The {\bf Erd\H{o}s number} of a person can be defined recursively as follows:
Paul Erd\H{o}s has Erd\H{o}s number 0. A person other than Erd\H{o}s
himself has Erd\H{o}s number one more than the smallest Erd\H{o}s number
among his/her coauthors. If none of the person's coauthors have an Erd\H{o}s
number, then neither does that person. Equivalently, the Erd\H{o}s number of
a person is the shortest-path distance between that person and Paul Erd\H{o}s
in the graph where there is an edge between two people if they are among
the authors of a paper in mathematics. We refer to this graph as the {\bf
collaboration graph}. The American Mathematical Society provides an online
tool\footnote{\url{https://mathscinet.ams.org/mathscinet/freeTools.html?version=2}}
for computing the shortest-path distance between any two mathematicians
in the collaboration graph. Oakland University hosts the Erd\H{o}s number
project\footnote{\url{https://oakland.edu/enp/}} which provides many
interesting facts and data about the Erd\H{o}s number.

The shortest-path distance is a measure of closeness between nodes in
a graph. One may argue, however, that it is an inadequate measure of
closeness between collaborators. For instance, it is reasonable to say
that the more joint articles between the two people, the closer they are
as collaborators. However, such a natural idea is completely ignored by the
shortest-path distance. The ratio of the number of joint articles to the total
number of publications between two authors is another piece of information
that can be used in measuring distances of collaborators. These considerations
suggest one to view the collaboration graph as a weighted graph rather than
just a simple graph. In a finite weighted graph, the lightest-path distance
between two nodes in the same connected component is the minimum path weight
between the two nodes. The resistance distance is another metric on weighted
graphs in which every path between two nodes contributes some decrement
of the distance between them. The idea is that the more paths there are
connecting the two nodes, the less ``resistance'' there is to travel from
one to the other and hence the closer they are. Finding effective resistance
between nodes in an electric circuit is certainly familiar to engineers,
while viewing it as a metric on graphs is no strange business to graph
theorists either. We refer the reader to the book~\cite{mgt} by Bollob\'{a}s
and the article~\cite{ael} by Shapiro for more information. More recently,
Chebotarev constructed a family of graph-geodetic distances~\cite{graph-geo}
in which the lightest-path distance and the resistance distance correspond
to the two extreme cases of the parameter. Using the resistance distance to
measure closeness of collaborators was considered in~\cite{ren},~\cite{ren2}
and~\cite{ren3}. Some non-metrical generalizations of the Erd\H{o}s number that
measure proximity between nodes in weighted networks was proposed and studied
in~\cite{gen}. Since it is of secondary interest to us in this article, we
mention in passing that algorithms for computing the lightest-path distance
are well-known~\cite{algor} and those for computing resistance distance
have also been widely studied. The article~\cite{recursive} contain more
references on this topics.

There is one aspect of using either of the lightest-path distance or the
resistance distance to define Erd\H{o}s numbers that is unsatisfactory to us,
namely, the relative closeness between authors given by these metrics may
contradict the one given by the shortest-path distance: $A$ may be closer to
$B$ than to $C$ according the shortest-path distance but the exact opposite
may be true for either of the two distances aforementioned. Because of this,
we set our goal to finding metrics on the collaboration graph that, in some
fashion, take into account the number of joint articles but not contradicting
the shortest-path distance.

The rest of this article is organized as follows: in
Section~\ref{sec:metrics-and-norms} we introduce a notion called
monoid norm and use it to unify various constructions of metrics. In
Section~\ref{sec:refinement} we introduce the refinement relation on
functions defined on a Cartesian product with codomain a totally ordered
set.  We then show how to produce refinements of a metric by another
metric. We also identify a condition under which the refining process can be
iterated. Section~\ref{sec:metric+proximity} is devoted to a particular kind
of refinement of the shortest-path distance. Unlike the constructions given
in Section~\ref{sec:refinement}, the additional functions use in the refining
process are no longer metrics. But it is crucial that the metric being refined
is the shortest-path distance. Lastly, in Section~\ref{sec:computation},
we compute the new Erd\H{o}s numbers of a few mathematicians corresponding
to different refinements of the shortest path distance. We end the article
by proposing another the edge weight function which seems to be appropriate
for the purpose of refining the Erd\H{o}s number.

\section{Monoid norms} \label{sec:metrics-and-norms}

To produce metrics that fit our requirements set forth in the introduction,
we use several basic constructions of metrics. What seems to be new to us
here is the realization that all these constructions can be unified into a
single one. This led us to the following pair of notions. Let $(M, +, 0_M)$
be a monoid (written additively). We call a function $\mu$ from $M$ to $\Rr$
a {\bf monoid norm} if
\begin{enumerate}
  \item $\mu(x) \ge 0$ for any $x \in M$; 
  \item $\mu(x) = 0$ if and only if
  $x = 0_M$; and 
  \item $\mu\left( x+y \right) \le \mu(x) + \mu(y)$ for any
  $x,y \in M$.  (subadditivity)
\end{enumerate} 
A partially ordered monoid is a monoid $(M,+,0_M)$ equipped with a partial
order $\le_M$ on $M$ that respects translation, i.e. $x \le_M y$ implies
$x+z \le_M y+z$ for any $x,y,z \in M$. As an example, let $\Rr^n_{\ge 0}$
($n \ge 1$) be the set of $n$-tuples of non-negative real numbers. For $\vx
=(x_1,\ldots, x_n)$ and $\vx' = (x_1', \ldots, x_n')$ in $\Rr^n_{\ge 0}$,
let $\vx \le \vx'$ if $x_i \le x_i'$ in the usual order of real numbers for
each $1 \le i \le n$. The relation $\le$ thus defined is the product order on
$\Rr^n_{\ge 0}$. It is straightforward to check that $\Rr^n_{\ge 0}$ equipped
with component-wise addition and the product order is a partially ordered
monoid. We call a monoid norm $\mu$ on a partially ordered monoid $M$ {\bf
monotonic} if $x \le_M y$ implies $\mu(x) \le \mu(y)$. The reader will likely
recognize that the names of these notions are taken from their counterparts
for real vector spaces. The only difference is that the homogeneity property
of norms, that is $\mu(\alpha x) = |\alpha| \mu(x)$ for $\alpha \in \Rr$,
which does not make sense for monoids in general, is being dropped.
\begin{example}
  \label{ex:mmn}
The following functions are monotonic monoid norms (the first three functions
are defined for $\Rr_{\ge 0}$ and the last one is defined for $\Rr_{\ge 0}^n$
($n \ge 1$)):
\begin{enumerate}[(i)]
  \item $\mu(x) = \alpha x$ ($\alpha > 0$).

  \item $\mu(x) = \ceil{x}$ where $\ceil{x}$ denotes the least integer no
  smaller than $x$.

  \item $\mu(x) = \dfrac{x}{1+x}$.

  \item $\mu(x_1,\ldots, x_n) = x_1 + \cdots + x_n$.
\end{enumerate} 
We will verify the last function is
a monotonic monoid norm and leave the verification of the other three to the
reader. For any $\vx=(x_1,\ldots, x_n) \in \Rr_{\ge 0}^n$, since each $x_i$
is non-negative, $\mu(\vx) = \sum_{i=1}^n x_i \ge 0$ and $\mu(\vx) = 0$ if and
only if each $x_i = 0$, i.e. $\vx = \vec{0}$. Moreover, since $\vx + \vx' =
(x_1+x_1', \ldots, x_n+x_n')$, so $\mu(\vx+\vx')$ actually equals $\mu(\vx)
+ \mu(\vx')$. Thus, $\mu$ is a monoid norm on $\Rr^n_{\ge 0}$. Finally,
suppose $\vx \le \vx'$; that is, $x_i \le x_i'$ for each $1 \le i \le n$
and so $\mu(\vx) = \sum_{i=1}^n x_i \le \sum_{i=1}^n x_i' = \mu(\vx')$.
Therefore, $\mu$ is monotonic.
\end{example} 
Some monoid norms on $\Rr^n_{\ge 0}$ are clearly not restrictions of norms on
$\Rr^n$. One example is the ceiling function $x \mapsto \ceil{x}$. Another
example is the function defined by $\mu(\vec{0})=0$ and $\mu(\vx) = 1$ for
all $\vx \neq \vec{0}$. Both of them fail the homogeneity property for being
a norm. On the other hand, many familiar norms on $\Rr^n$, e.g. the $\ell_p$
norms, are monotonic on $\Rr_{\ge 0}^n$. In the literature, norms that are
monotonic on various orthants are studied under the name of orthant-monotonic
norms. Despite the fact that there are numerous characterizations of these
norms~\cite{absmono, charnorms}, we were unable to find in the literature an
explicitly given norm on $\Rr^n$ that is not monotonic on $\Rr_{\ge 0}^n$. So
it may be worthwhile to include a family of examples here.

\begin{minipage}[r]{.49\textwidth}
  \begin{example}
  \label{ex:non_orthant-monotonic} The function $\nu_n \colon \Rr^n \to \Rr$
  ($n \ge 2$) defined by \[
    \nu_n(\vx) = \max_{1\le k \le n}|x_k| + \sum_{i < j} |x_i -x_j|
  \] is a norm on $\Rr^n$ restricting to a monoid norm on $\Rr^n_{\ge 0}$
  that is not monotonic. For instance, $(0,2, \ldots, 2) \le (1,2, \ldots,
  2)$ in $\Rr^n_{\ge 0}$ and yet 
  \begin{align*}
    \nu_n(0,2, \ldots, 2) &= 2+2(n-1) > 2+(n-1)\\ &= \nu_n(1,2, \ldots, 2).
  \end{align*}
\end{example} 
\end{minipage} 
\begin{minipage}[c]{.5\textwidth}
  \centering \captionsetup{type=figure}
  \includegraphics[width=0.8\textwidth]{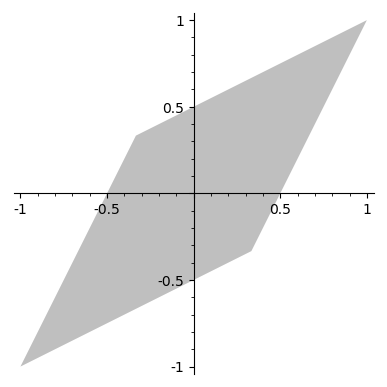} \captionof{figure}{The unit
  ball of $\nu_2$.} \label{fig:ub}
\end{minipage}

The following statement is the key in unifying various constructions of
metrics by the notion of monotonic monoid norms.  
\begin{theorem}
  Let $(E_i, d_i)$, $1 \le i \le n$, be metric spaces and $\mu$ be a monotonic
  monoid norm on $\Rr^n_{\ge 0}$. Then $\mu \circ \vec{d}$, where $\vec{d} =
  \prod_{i=1}^n d_i$, is a metric on the Cartesian product $\prod_{i=1}^n E_i$.
  \label{th:mmn-metric}
\end{theorem} 
\begin{proof}
  Since each $d_i$ is non-negative, the range of $\vec{d}$ is in $\Rr_{\ge
  0}^n$. Thus, the function $\mu \circ \vec{d}$ is well defined. Since $\mu$
  takes only non-negative values, so does $\mu \circ \vec{d}$ and since each
  $d_i$ is symmetric, $\vec{d}(\vec{p}, \vec{q}) = \vec{d}(\vec{q},\vec{p})$
  for any $\vec{p}, \vec{q} \in \prod_{i=1}^n E_i$.

  For $1 \le i \le n$, since $d_i(p_i,p_i) = 0$ for every $p_i \in E_i$,
  $\vd(\vp,\vp) = \vec{0}$ and hence $\mu \circ \vd(\vp,\vp) = 0$ for each
  $\vp \in \prod_{i=1}^n E_i$. Because $\mu$ only maps $\vec{0}$ to $0$, $\mu
  \circ \vd(\vp, \vq) = 0$ implies $\vd(\vec{p}, \vec{q}) = \vec{0}$. The
  last equation means $d_i(p_i,q_i) = 0$ ($1 \le i \le n$) and so $p_i=q_i$
  since each $d_i$ is a metric. Therefore, $\vec{p} = \vec{q}$.

  Finally, since each $d_i$ satisfies the triangle inequality, $\vd(\vp,
  \vq) \le \vd(\vp,\vr) + \vd(\vr, \vq)$ in the product order of $\Rr^n_{\ge
  0}$. It then follows from the monotonicity and the subadditivity of $\mu$
  that 
  \begin{align*}
    \mu \circ \vd(\vp,\vq) \le \mu \left( \vd(\vp, \vr) +\vd(\vr, \vq)
    \right) \le \mu \circ \vd(\vp,\vr) + \mu \circ \vd(\vr, \vq).
  \end{align*} 
  Thus, $\mu \circ \vd$ is indeed a metric on $\prod_{i=1}^nE_i$.  
\end{proof}

The next few propositions about metrics, in the light of
Theorem~\ref{th:mmn-metric}, are all consequences of the fact that the
functions in Example~\ref{ex:mmn} are monotonic monoid norms.
\begin{proposition}
  If $d$ is a metric, then so are $\alpha d$ $(\alpha > 0)$, $\ceil{d}$
  and $d_{\flat} := \dfrac{d}{1+d}$.  \label{p:various}
\end{proposition}

\begin{proposition}
  Let $(E_i,d_i)$, $1 \le i \le n$, be metric spaces, then the map $d$
  defined by $d(\vec{p}, \vec{q}) = \sum_{i=1}^n d_i(p_i, q_i)$ is a metric
  on $\prod_{i=1}^n E_i$. \label{p:prod-metric}
\end{proposition} 
\begin{corollary}
  The sum of finitely many metrics on a set $E$ is metric on $E$.
  \label{c:finite_sums}
\end{corollary} 
\begin{proof}
  Let $d_1, \ldots, d_n$ be metrics on $E$. By Proposition~\ref{p:prod-metric},
  $d(\vec{p}, \vec{q}) = \sum_{i=1}^n d_i(p_i,q_i)$ is a metric on $E^n$
  and hence on its diagonal which can be identified with $E$ itself via
  $p \mapsto \iota(p)=(p,\ldots, p)$. Thus, $d(p,q) = d \left( \iota(p),
  \iota(q) \right) = \sum_{i=1}^n d_i(p,q)$ is a metric on $E$.
\end{proof}

\section{Refinements} \label{sec:refinement} 

We propose the following notion for functions from a Cartesian product to
a totally ordered set.
\begin{definition}
  \label{df:refine} Let $f$ and $f'$ be functions from a Cartesian product
  $X$ to a totally ordered set $(T, \le)$. We say that $f'$ {\bf refines} $f$
  if $f'(\vx) < f'(\vx')$ whenever $f(\vx) < f(\vx')$ for any $\vx,\vx' \in
  X$ with all but one coordinate the same. We write $f' \preceq f$ if $f'$
  refines $f$.
\end{definition}
The relation $\preceq$ is clearly reflexive and transitive. It is however not
anti-symmetric. For example, the identity function of $\Rr$ and twice this
function refine each other. More generally, if $\sigma$ is an order-preserving
embedding from the range of $f$ to $T$ then $f$ refines $\sigma\circ f$
and vice versa. By a {\bf refinement class}, we mean an equivalence class
of the equivalence relation in which two functions are equivalent if they
refine each other. Note that the refinement relation $\preceq$ induces a
partial order, still denoted by $\preceq$, on refinement classes.

For unary functions, $f'$ refines $f$ means $f'$ never contradicts $f$
on strict inequalities. For binary functions, $f'$ refines $f$ if and
only if for all $(x_1,x_2) \in X$, $f'(x_1,\phd)$ refines $f(x_1,\phd)$
and $f'(\phd,x_2)$ refines $f(\phd,x_2)$. Metrics are symmetric binary
functions, so for metrics $d$ and $d'$ on the same set $E$, $d'$ refines $d$
means for any $p,q,r \in E$, $d'(p,q) < d'(p,r)$ whenever $d(p,q) < d(p,r)$
or, equivalently, $d'(p,\phd)$ refines $d(p,\phd)$ for each $p \in E$.

There is also a graph theoretic interpretation of refinement. Suppose $X$
is the Cartesian product of a family $\{X_i\}_{i\in I}$ of sets. Let $K_i$
be the complete graph with vertex set $X_i$ ($i \in I$). Let $G_X$ be the
Cartesian product of the family $\{K_i\}_{i \in I}$ of graphs. In other words,
$G_X$ is the graph with vertex set $X$ where two elements of $X$ are adjacent
in $G_X$ if and only if they differ at exactly one $i \in I$. Note that $G_X$
is connected if the index set $I$ is a finite set. On the other hand, $G_X$
is disconnected if each $X_i$ is nonempty and $|X_i| \ge 2$ for infinitely
many $i \in I$. A function $f$ from $X$ to a totally ordered set $T$ can
be viewed as a ``potential function'' on $G_X$ that gives a directed graph
structure on $G_X$: an edge of $G_X$ between $\vx$ and $\vx'$ becomes an
arc from $\vx$ to $\vx'$ if $f(\vx') \le f(\vx)$. We denote the resulting
directed graph by $G_X(f)$. In this setup, $f'$ refines $f$ simply means
$G_X(f')$ is a directed subgraph of $G_X(f)$. We say that a function from $X$
to $T$ is locally constant if it is constant on each connected component
of $G_X$. It is clear that any function from $X$ to $T$ refines a locally
constant function. Consequently, the locally constant functions form the
greatest element in the partial order $\preceq$. The directed graph of this
class, denoted by $G_X(*)$, is obtained by replacing each edge in $G_X$
by a pair of opposing arcs. On the other hand, the refinement class of a
proper coloring of $G_X$, i.e. a function that assigns distinct elements of
$T$ to neighbors in $G_X$, is minimal in $\preceq$. Let us illustrate these
ideas by the following simple example.
\begin{example}
  Let $X$ be the Cartesian power $\{0,1\}^2$ and $T$ be the totally ordered
  set $\{ \cz , \cy \}$ with $\cz < \cy$. Then $G_X$ is the square graph.
  The diagram in Figure~\ref{fig:por} shows the directed graphs and the
  relation between the refinement classes of four functions from $X$ to $T$.
  \tikzset{
    auto,node distance =1 cm and 1 cm,semithick,
    state/.style ={circle, draw, minimum width = 0.7 cm},
    point/.style = {circle, draw, inner sep=0.04cm,fill,node contents={}},
    ma/.style={decoration={
	markings,
    mark=at position .5 with {\arrow[>=Latex]{>}}}, postaction={decorate}},
    ca/.style={decoration={
	markings,
    mark=at position .6 with {\arrow[>=Latex]{>}}}, postaction={decorate}},
    bidirected/.style={Latex-Latex},
    el/.style = {inner sep=2pt, align=left, sloped}
}
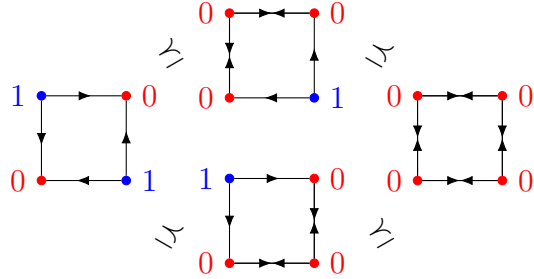
\begin{figure}[ht]
  \centering
\begin{tikzpicture}
  \node[red] (00) at (0,0) [label=left:\cz, point];
  \node[blue] (10) [right =of 00] [label=right:\cy, point];
  \node[red] (11) [above =of 10] [label=right:\cz, point];
  \node[blue] (01) [above =of 00] [label=left:\cy, point];

  \node (r1) at (1.7,1.7) [rotate=30] {$\preceq$};
  \node (r2) at (1.7, -0.7) [rotate=-30] {$\preceq$};
  \node (r3) at (4.5, 1.7) [rotate=-30] {$\preceq$};
  \node (r4) at (4.5, -0.7) [rotate=30] {$\preceq$};

  \node[red] (00a) at (2.5,1.1) [label=left:\cz, point];
  \node[blue] (10a) [right =of 00a] [label=right:\cy, point];
  \node[red] (11a) [above =of 10a] [label=right:\cz, point];
  \node[red] (01a) [above =of 00a] [label=left:\cz, point];

  \node[red] (00b) at (2.5,-1.1) [label=left:\cz, point];
  \node[red] (10b) [right =of 00b] [label=right:\cz, point];
  \node[red] (11b) [above =of 10b] [label=right:\cz, point];
  \node[blue] (01b) [above =of 00b] [label=left:\cy, point];

  \node[red] (00c) at (5,0) [label=left:\cz, point];
  \node[red] (10c) [right =of 00c] [label=right:\cz, point];
  \node[red] (11c) [above =of 10c] [label=right:\cz, point];
  \node[red] (01c) [above =of 00c] [label=left:\cz, point];

  \path (10) edge[ca] (00);
  \path (10) edge[ca] (11); 
  \path (01) edge[ca] (11);
  \path (01) edge[ca] (00);

  \path (10a) edge[ca] (00a);
  \path (10a) edge[ca] (11a); 
  \path (11a) edge[ma] (01a); \path (01a) edge[ma] (11a);
  \path (01a) edge[ma] (00a); \path (00a) edge[ma] (01a);

  \path (00b) edge[ma] (10b); \path (10b) edge[ma] (00b);
  \path (10b) edge[ma] (11b); \path (11b) edge[ma] (10b); 
  \path (01b) edge[ca] (11b);
  \path (01b) edge[ca] (00b);

  \path (00c) edge[ma] (10c); \path (10c) edge[ma] (00c);
  \path (10c) edge[ma] (11c); \path (11c) edge[ma] (10c); 
  \path (11c) edge[ma] (01c); \path (01c) edge[ma] (11c);
  \path (01c) edge[ma] (00c); \path (00c) edge[ma] (01c);

\end{tikzpicture}
  \caption{Part of the refinement order}
  \label{fig:por}
\end{figure}
 
  The constant function determines the greatest class and the
  proper coloring determines a minimal class in the refinement order.
\end{example}

Let $(T, +, 0_T, \le)$ be an ordered abelian group. For $t \in T$, we write
$|t|$ for $\max\{t,-t\}$. Let $A$ and $B$ be subsets of $T$, we write $|A|$
for the set $\{ |a| \colon a \in A\}$ and write $A > B$ if $a > b$ for every
$a \in A$ and $b \in B$. If $f$ is a function from a Cartesian product $X$ to
$T$, we write $\Delta(f)$ for the set $\{f(\vx')-f(\vx) \colon \text{$\vx$ and
$\vx'$ are adjacent in } G_X\}$ and call it the {\bf difference set} of $f$.
We write $\Delta^+(f)$ for the set of positive elements of $\Delta(f)$. Note
that $\Delta(f) = - \Delta(f)$ and so $\Delta^+(f) = |\Delta(f)| \setminus
\{0_T\}$. Consequently, $\Delta^+(f)$ is empty if and only if $f$ is constant
on neighbors in $G_X$ if and only if $f$ is locally constant.

\begin{theorem}
  \label{th:oag-refines} Let $(T,+,0_T, \le)$ be an ordered abelian group.
  If $f$ and $g$ are functions from a Cartesian product $X$ to $T$ with
  $\Delta^+(f) > \Delta^+(g)$, then $f+g$ refines $f$.
\end{theorem}
\begin{proof}
  Suppose $\vx$ and $\vx'$ in $X$ differ by one coordinate and that $f(\vx')
  > f(\vx)$. Since $|g(\vx)-g(\vx')|$ is either $0_T$ or in $\Delta^+(g)$,
  then by assumption it is smaller than $f(\vx')-f(\vx)$. Thus,
  \begin{align*}
    (f+g)(\vx') - (f+g)(\vx) &= (f(\vx') - f(\vx)) - (g(\vx) - g(\vx')) \\
    &\ge  (f(\vx') - f(\vx)) - |g(\vx) - g(\vx')| > 0_T.
  \end{align*}
  This shows that $f+g$ refines $f$.
\end{proof}
Theorem~\ref{th:oag-refines} manifests a simple idea: a function can be
refined by adding a small (relative to the function) perturbation. The next
few propositions should convince the reader that this idea is central to
our treatment of refinements.

\begin{proposition}
  Let $f, g$ be real-valued functions on a Cartesian product $X$. Then
  $f + \alpha g$ refines $f$ for any $\alpha \in \Rr$ such that
  $|\alpha|\Delta^{+}(g) < \Delta^+(f)$.  \label{p:real-valued}
\end{proposition}
\begin{proof}
  The proposition is trivial for $\alpha =0$. For $\alpha \neq 0$,
  \begin{align*}
    \Delta^+(\alpha g) &= |\Delta(\alpha g)| \setminus \{0\}
    = |\alpha||\Delta(g)| \setminus \{0\} = |\alpha|\left( |\Delta(g)|
    \setminus \{0\} \right) = |\alpha|\Delta^+(g), \label{eq:factor}
  \end{align*}
  and so it follows from Theorem~\ref{th:oag-refines}.
\end{proof}
Readers who are familiar with ordered fields will certainly recognize the
validity of Proposition~\ref{p:real-valued} and Corollary~\ref{c:finite_case}
below in that setting because the properties of $\Rr$ being used in their
proofs are those of an ordered field.
\begin{corollary}
  Suppose the difference sets of the functions $f$ and $g$ in
  Proposition~\ref{p:real-valued} are finite. Then there exists $\alpha_0 \in
  (0, \infty]$, depending on both $f$ and $g$, such that $f + \alpha g$ refines
  $f$ for any $\alpha$ with $|\alpha| < \alpha_0$.  \label{c:finite_case}
\end{corollary}
\begin{proof}
  The assumption is equivalent to the sets $\Delta^+(f)$ and $\Delta^+(g)$
  being finite. If either of them is empty, then either $f$ or $g$ is
  constant on neighbors in $G_X$ and so $f + \alpha g$ refines $f$ for any
  $\alpha \in \Rr$. We then establish the corollary by taking $\alpha_0 =
  \infty$. If $\Delta^+(f)$ and $\Delta^+(g)$ are nonempty, then both $m
  := \min \Delta^+(f)$ and $M := \max \Delta^+(g)$ exist and are positive
  numbers. So for any $\alpha$ with $|\alpha| < \alpha_0:= m/M$,
  \[
    |\alpha| \Delta^+(g) \le |\alpha|M < m \le \Delta^{+}(f).
  \]
  Thus, the corollary follows from Proposition~\ref{p:real-valued}.
\end{proof}
We now turn to the case in which the functions
involved are metrics.
\begin{proposition}
  Let $d,d_0$ be metrics on a set $E$ such that $d$ is $\lambda \Zz$-valued
  and $d_0$ is $[0,\lambda)$-valued for some positive $\lambda$. Then $d+d_0$
  is a metric refining $d$. In particular, if $d$ takes integer values then
  $d + d'_{\flat}$ is a metric refining $d$ for any metric $d'$ on $E$.
  \label{p:metric_refine}
\end{proposition}
\begin{proof}
  The sum of two metrics is a metric (Proposition~\ref{p:prod-metric}) so
  $d+d_0$ is a metric on $E$. Since
  \[
    \Delta^+(d_0) \subseteq (0,\lambda) < \Delta^+(d) \subseteq \lambda\Nn,
  \]
  it follows from Theorem~\ref{th:oag-refines} that $d+d_0$ refines $d$. The
  special case follows because $d'_{\flat}$ is a $[0,1)$-valued metric for
  any metric $d'$ (Proposition~\ref{p:various}).
\end{proof}
The metric $d+d_0$ in Proposition~\ref{p:metric_refine} need not be
taking values in a discrete set. So the process of refining metrics
given by that proposition cannot be iterated in general. However,
Corollary~\ref{c:finite_case} can be used instead if the metrics involved
have finite difference sets. This happens if the metric space is finite or
more generally when the metrics involved take only finitely many values. For
instance, given a sequence $d_1, \ldots, d_n$ of metrics on a finite graph
$G$ we construct another sequence of metrics $(d'_i)_{1 \le i \le n}$ on $G$
as follows: let $d'_1 = d_1$ and suppose $d'_i$ has been constructed for
some $1 \le i \le n$. Then according to Corollary~\ref{c:finite_case},
we can construct a refinement $d'_{i+1}$ of $d'_i$ by adding to it a
suitable positive multiple of $d_{i+1}$. By Proposition~\ref{p:various}
and Proposition~\ref{p:prod-metric}, $d'_{i+1}$ is still a metric. Thus,
the construction of $(d'_i)_{1 \le i \le n}$ is completed by induction. All
metrics in the original sequence contribute to $d'_n$ the last metric in the
new sequence. Moreover, the new metrics respect the sequence's order in the
sense that for any $1 \le i \le n-1$ the relative closeness of vertices in
$G$ determined by $d'_i$ will not be contradicted by $d'_{i+1}$.

\section{Refinements of the Shortest-Path Distance}
\label{sec:metric+proximity}

We now focus on refining a particular metric---the shortest-path
distance. General results about refinements in the previous section still
apply. The difference here is that the refinements are obtained by modifying
the shortest-path distance by functions that are not metrics themselves and
the challenge is to come up with the right kind of functions so that the
resulting refinements are still metrics.

Let $G$ be a finite simple connected graph. Let $w$ be a non-negative real
function on the edges of $G$. The {\bf product weight} of a path $\gamma$
in $G$, denoted by $p(\gamma)$, is defined to be the product of the weights
of its edges. That is, $p(\gamma) = \prod_{e \in \gamma} w(e)$. For vertices
$u$ and $v$ of $G$, let
\[
  \pi(u,v) = 
  \begin{cases}
    \min\left\{\displaystyle \frac{p(\gamma)}{p(\gamma)+1} \colon \gamma\
    \text{is a shortest path between $u$ and $v$.}\right\} & \text{if}\
    u \neq v; \\ 0 & \text{if}\ u = v.
  \end{cases}
\] 
Clearly $\pi(u,v)$ is symmetric and takes values in $[0,1)$ but $\pi$ is not
a metric since $\pi(u,v)$ could be 0 for distinct $u$ and $v$. Moreover,
$\pi$ does not satisfy the triangle inequality in general. For instance,
consider the following triangle with the indicated edge weights:
\tikzset{
    auto,node distance =1 cm and 1 cm,semithick,
    state/.style ={circle, draw, minimum width = 0.7 cm},
    point/.style = {circle, draw, inner sep=0.04cm,fill,node contents={}},
}
\begin{center}
\begin{tikzpicture}
  \node (u) at (-1,0) [label=left:$u$, point];
  \node (v) at (0,1.732) [label=north:$v$, point];
  \node (w) at (1,0) [label=right:$w$, point];

  \path (u) edge node {$\frac{1}{4}$} (v); 
  \path (w) edge node {$\frac{3}{4}$} (u); 
  \path (v) edge node {$\frac{1}{4}$} (w);
\end{tikzpicture}
\end{center}
The triangle inequality is violated because $\pi(u,w) = 3/7 > 2/5 = \pi(u,v)
+ \pi(v,w)$. In the following, let $d_s$ denote the shortest-path distance
and let $d^{\pi}_w$ be $d_s + \pi$.

\begin{theorem}
  $d^{\pi}_w$ is a metric on $G$.  \label{th:metric+proximity}
\end{theorem} 
\begin{proof}
  Only the triangle inequality deserves a proof. If it fails for $d^{\pi}_w$,
  then there exist vertices $x,y,z$ of $G$ such that $d^{\pi}_w(x,y) >
  d^{\pi}_w(x,z) + d^{\pi}_w(z,y)$. That is,
  \begin{equation}
    d_s(x,y) +\pi(x,y) > d_s(x,z) + \pi(x,z) + d_s(z,y) +\pi(z,y).
    \label{ineq:contradiction}
  \end{equation} 
  Since $d_s$ itself satisfies the triangle inequality and $\pi$ takes values
  in $[0,1)$, we obtain the following inequalities by rearranging the terms
  in~\eqref{ineq:contradiction}: 
  \begin{equation}
    1 \ge \pi(x,y) - \pi(x,z) -\pi(z,y) > d_s(x,z) + d_s(z,y) -d_s(x,y)
    \ge 0. \label{ineq:<1}
  \end{equation} 
  Because $d_s$ takes integer values, it follows from~\eqref{ineq:<1} that
  \begin{align}
    d_s(x,y) = d_s(x,z) + d_s(z,y). \label{eq:=}
  \end{align} 
  Choose a shortest path $\gamma_{xz}$ between $x$ and $z$ such that
  $p(\gamma_{xz})/(p(\gamma_{xz})+1)$ realizes the value $\pi(x,z)$. Choose
  $\gamma_{zy}$ analogously. Equation~\eqref{eq:=} implies the walk $\gamma$
  obtained by concatenating $\gamma_{xz}$ and $\gamma_{zy}$ has length
  $d_s(x,y)$ and because every walk contains a path with the same ends,
  $\gamma$ must be a shortest path between $x$ and $y$. From this and the
  fact that $p$ is a nonnegative function, we conclude that
  \begin{align}
    \begin{split} \pi(x,y) &\le \frac{p(\gamma)}{p(\gamma)+1} =
    \frac{p(\gamma_{xz})p(\gamma_{zy})}{p(\gamma_{xz})p(\gamma_{zy})+1}\\
    &\le \frac{p(\gamma_{xz})}{p(\gamma_{xz})+1}
    + \frac{p(\gamma_{xz})}{p(\gamma_{xz})+1} = \pi(x,z) + \pi(z,y).
    \end{split} 
    \label{ineq:w}
  \end{align} 
  However, the inequalities in~\eqref{ineq:w} are in contradiction with those
  in~\eqref{ineq:<1}. Thus, $d_w^\pi$ must satisfy the triangle inequality
  as well.
\end{proof} 
Since  $\Delta^+(\pi) \subseteq (0,1) < \Delta^+(d_s) \subseteq \Nn$,
$d^{\pi}_w$ refines $d_s$ by Theorem~\ref{th:oag-refines}. This fact together
with Theorem~\ref{th:metric+proximity} imply: 
\begin{proposition}
  $d^{\pi}_w$ is a metric refining $d_s$.  \label{p:refine2}
\end{proposition} 
Theorem~\ref{th:metric+proximity} can be generalized in a number of
ways. First, the finiteness assumption on $G$, which is harmless to our
applications because the collaboration graph itself is finite, can be
dropped.  Its only use is to guarantee the minimum in the definition of
$\pi$ exists. For that matter we can assume between any two vertices of $G$
there are only finitely many paths, or even just finitely many shortest
paths. In fact, by taking infimum instead of minimum in the definition of
$\pi$ we can drop these assumptions altogether. However, with that change
we can no longer guarantee a value of $\pi$ is realized by a path, yet for
any distinct vertices $u, v$ and $\varepsilon > 0$, there will be a shortest
path $\gamma_{uv}$ between $u$ and $v$ with $p(\gamma_{uv})/(p(\gamma_{uv})
+ 1) < \pi(u,v)+\varepsilon$. Since $\pi$ still takes values in $[0,1)$,
the inequalities in~\eqref{ineq:<1} continue to hold. Therefore, the walk
$\gamma = \gamma_{xz}\gamma_{zy}$ must again be a shortest path between $x$
and $y$, and so
\begin{align*}
  \pi(x,y) &\le \frac{p(\gamma)}{p(\gamma)
  +1} \le \frac{p(\gamma_{xz})}{p(\gamma_{xz})+1}
  + \frac{p(\gamma_{zy})}{p(\gamma_{zy})+1} < \pi(x,z) + \pi(z,y) +
  2\varepsilon.  \label{ineq:e}
\end{align*} 
This establishes $\pi(x,y) \le \pi(x,z)+\pi(z,y)$ and that is what needed
to complete the proof.

Second, the assumption on $G$ being connected can also be removed. We extend
$d_s$ and $\pi$ by declaring $d_s(u,v) = \infty$ and $\pi(u,v) = 0$ whenever
$u,v$ are in different components of $G$. With that change $d_s$ is no longer
real-valued and hence not a metric in the strict sense. However, if one
adopts the usual conventions: $\infty + \infty = \infty$, $r < \infty$ and
$\infty + r = \infty$ $(r \in \Rr)$, then $d_s$ still satisfies the triangle
inequality. Moreover, $d^{\pi}_w(u,v) = d_s(u,v) +\pi(u,v)=\infty$ if and
only if $u$ and $v$ are in different components of $G$.  It follows that
$d^{\pi}_w$ can only fail the triangle inequality because it fails on some
component of $G$, i.e. there exist vertices $x,y,z$ in the same component of
$G$ with $d^{\pi}_w(x,y) > d^{\pi}_w(x,z) + d^{\pi}_w(z,y)$. Hence, the proof
of Theorem~\ref{th:metric+proximity} goes through without any modification.

\section{Some Refined Erd\H{o}s numbers} \label{sec:computation}

In this section we propose two different refinements of the shortest-path
distance of the collaboration graph\footnote{It is the graph given by
the MathSciNet database as of the time of submission of this article.}
$C$. To demonstrate how these refinements differentiate people who have
the same Erd\H{o}s number, we compute the new Erd\H{o}s numbers of a few
mathematicians. For an edge $e$ in $C$, let its weight $w(e)$ be the reciprocal
of $j(e)$ the number of joint articles between the two ends of $e$. As in
Section~\ref{sec:metric+proximity}, $d_s$ denotes the shortest-path distance
and $d^{\pi}_w$ denotes the refinement $d_s+\pi$. The {\bf sum weight} of
a path $\gamma$, denoted by $s(\gamma)$, is defined to be the sum of the
weights of its edges. That is, $s(\gamma) = \sum_{e \in \gamma} w(e)$. The
{\bf lightest-path distance} on $C$ is then given by the function
\begin{equation}
  \sigma(u,v) = 
  \begin{cases} \min\left\{s(\gamma) \colon \text{$\gamma$ is a path
    between $u$ and $v$.}\right\} &\text{if $u \neq v$}; \\ 0 &\text{if $u
    = v$}.
  \end{cases}
\end{equation}
We write $d^{\sigma}_w$ for the refinement $d_s + \sigma_{\flat}$ of $d_s$
and use $\pEN(x)$ and $\sEN(x)$ to denote the Erd\H{o}s number of $x$ defined
by $d^{\pi}_w$ and $d^{\sigma}_w$, respectively\footnote{EN
  stands for Erd\H{o}s Number. The 'p' and the 's' in the notation indicate
  the refinements are coming from a product and a sum of edge weights,
  respectively.}.

We rely on the data provided by MathSciNet and the Erd\H{o}s Number Project for
computations. By definition, Erd\H{o}s himself has pEN and sEN zero. Andr\'as
S\'ark\"ozy is the most frequent collaborator of Erd\H{o}s and vice versa. They
have 62 joint papers. It follows that S\'ark\"ozy has the smallest positive
refined Erd\H{o}s numbers of all. The pEN and the sEN of S\'ark\"ozy are both
\[
  1 + \frac{\frac{1}{62}}{1+\frac{1}{62}}= \frac{64}{63} \approx 1.016.
\]
The most frequent collaborator of Andr\'as Hajnal is Erd\H{o}s. They
have 57 joint publications. Consequently, both pEN and sEN of Hajnal are
$58/57 \approx 1.017$. Hajnal is the second most frequent collaborator of
Erd\H{o}s. He also has the second smallest positive pEN and sEN.

Christian Mauduit has Erd\H{o}s number 1 and has two joint articles with
Erd\H{o}s. So his pEN is $4/3$. Mauduit's most frequent collaborator
is S\'ark\"ozy, they have 41 joint papers. His second most frequent
collaborator is Jo\"el Rivat. They co-authored 16 articles. The path
Mauduit--S\'ark\"ozy--Erd\H{o}s has sum weight $1/41 + 1/62 = 103/2542$
which is less than the weight of the edge (1/16) between Mauduit and
Rivat. Therefore, it must be the lightest path between Mauduit and
Erd\H{o}s. Thus,
\[
  \sEN(\text{Mauduit}) = 1+ \frac{\frac{103}{2542}}{1+\frac{103}{2542}} =
  \frac{2748}{2645} \approx 1.039.
\]

Istv\'{a}n Juh\'{a}sz has Erd\H{o}s number 2. The six shortest paths between
Juh\'{a}sz and Erd\H{o}s go through Andr\'{a}s Hajnal, Peter Hamburger,
Kenneth Kunen, Menachem Magidor, Mary Ellen Estill Rudin and Saharon
Shelah, respectively. We organize the information given by these paths into
Table~\ref{tab:JE}.
\begin{table}[ht]
  \centering
  \begin{tabular}{|c|c|c|c|c|c|c|}
    \hline
    $j(\text{Juh\'asz}, x)$ & 32 & 1 & 5 & 1 & 1 & 13\\
    \hline
    $x$ & Hajnal & Hamburger & Kunen & Magidor & Rudin & Shelah \\
    \hline
    $j(x,\text{Erd\H{o}s})$ &57 &1 &1 &1 &1 &3 \\
    \hline
  \end{tabular}
  \caption{The shortest paths between Juh\'asz and Erd\H{o}s} 
  \label{tab:JE}
\end{table}
From the table it is clear that
\[
  \pEN(\text{Juh\'asz}) = 2 +
  \frac{\frac{1}{57}\frac{1}{32}}{1+\frac{1}{57}\frac{1}{32}} =
  \frac{3651}{1825} \approx 2.001.
\]
Given the form of data available to us it is harder to compute the sEN of
Juh\'{a}sz as it is more difficult to determine which paths between Juh\'{a}sz
and Erd\H{o}s are the lightest. First, the path Juh\"asz--Hajnal--Erd\H{o}s
gives $\frac{1}{57} + \frac{1}{32} = \frac{89}{1824} \approx 0.049$ as an
upper bound of the lightest-path distance between Juh\'asz and Erd\H{o}s. Since
$1824/89 \approx 20.5$, we only need to examine the collaborators of Juh\'asz,
besides Hajnal, who have at least 21 joint articles with Juh\'asz. Only
two mathematicians, namely Lajos Soukup (with 29 joint papers) and Zolt\'an
Szentmikl\'ossy (with 44 joint papers) meet this requirement. Solving the
inequality
\[
  \frac{89}{1824} > \frac{1}{44} + \frac{1}{n}
\]
yields $n > 38$. So in order for a lightest path between Juh\'asz and
Erd\H{o}s to go through either Szentmikl\'ossy or Soukup, each of them needs
a collaborator other than Juh\'asz with at least 39 joint articles. It turns
out that the most frequent collaborator of both Soukup and Szentmikl\'ossy
is Juh\'asz and their second most frequent collaborators are each other. They
have 24 joint articles. From this we conclude that Juh\'asz--Hajnal--Erd\H{o}s
is the unique lightest path between Juh\'asz and Erd\H{o}s. Consequently,
\[
  \sEN(\text{Juh\'asz}) = 2+ \frac{\frac{89}{1824}}{1+\frac{89}{1824}}
  = \frac{3915}{1913} \approx 2.047.
\]
We summarize this information\footnote{Here an edge is labeled not
by its weight but by the number of joint articles of its ends.} in
Figure~\ref{fig:isg}.
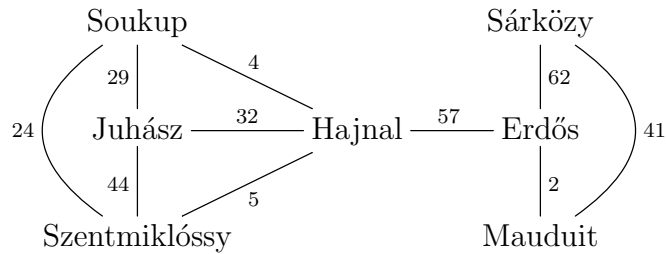
\begin{figure}[ht]
\[
  \xymatrix{ 
    \text{Soukup} \ar@{-}[d]_{29} \ar@{-}@/_3pc/[dd]_{24} \ar@{-}[dr]^{4} & &
    \text{S\'ark\"ozy} \ar@{-}[d]^{62} \ar@{-}@/^3pc/[dd]^{41} \\
    \text{Juh\'asz} \ar@{-}^{32}[r] & \text{Hajnal} \ar@{-}^{57}[r] & \text{Erd\H{o}s}\\
    \text{Szentmikl\'ossy} \ar@{-}^{44}[u] \ar@{-}[ur]_{5} & & \text{Mauduit}
    \ar@{-}[u]_{2}
  }
\]
\caption{A subgraph of the collaboration graph $C$.} \label{fig:isg}
\end{figure}
Moreover, Soukup and Szentmikl\'ossy both have 10 joint articles with their
third most frequent collaborators (J\'anos Gerlits for Szentmikl\'ossy
and Saharon Shelah for Soukup). With this additional information, we can
compute their sEN's and pEN's. We skip the details here but summarize in
Table~\ref{tab:EN1} the new Erd\H{o}s numbers of the mathematicians appearing
in Figure~\ref{fig:isg}.

We conclude this article by suggesting another edge weight function for the
collaboration graph that seems appropriate for the purpose of measuring
closeness between authors: take the weight of an edge to be the ratio
$j(e)/t(e)$ where $j(e)$ is the number of joint articles and $t(e)$ is the
total number of articles published by the ends of $e$. A more sophisticated
version of this weighting function that takes the types of publication into
account has been considered recently in~\cite{ren3}. We leave the computations
of the corresponding refined Erd\H{o}s numbers to the interested reader.
\begin{table}
  \centering
  \begin{tabular}{|c|c|c|c|}
    \hline 
    $x$ & EN & $\pEN(x)$ & $\sEN(x)$ \\ 
    \hline
    Erd\H{o}s & 0 & 0 & 0 \\
    \hline
    S\'ark\"ozy & 1 & $64/63 \approx 1.016$ & $64/63 \approx 1.016$ \\ 
    \hline 
    Hajnal & 1 & $59/58 \approx 1.017$ & $59/58 \approx 1.017$ \\ 
    \hline 
    Mauduit & 1 & $4/3 \approx 1.333$ & $2748/2645 \approx 1.039$\\ 
    \hline 
    Juh\'asz & 2 & $3651/1825 \approx 2.001$ & $3915/1913 \approx 2.047$ \\ 
    \hline
    Szentimikl\'ossy & 2& $573/286 \approx 2.003$ & $44433/21499 \approx
    2.067$ \\
    \hline
    Soukup & 2 & $459/229 \approx 2.004$ & $119007/57301 \approx 2.077$ \\
    \hline
  \end{tabular}
  \caption{Refined Erd\H{o}s Numbers of a few mathematicians.} \label{tab:EN1}
\end{table}

\section*{Acknowledgments} %%% ADD HERE YOUR THANKS AND ACKNOWLEDGMENTS
We thank Serban Raianu for carefully reading an earlier version of this
article and for bringing~\cite{gen} to our attention. We also thank him and
the referees for valuable advice which helped improve the presentation.

{\footnotesize
\bibliographystyle{plain} 
\bibliography{gd} 
}

%%%%%%%% AUTHORS' INFORMATION. DELETE/ADD AUTHORS AS NEEDED
{\footnotesize  
\medskip
\medskip
\vspace*{1mm} 
 
\noindent {\it Kayla Lock}\\  
Arizona State University\\
1151 S Forest Ave, Tempe, AZ 85281\\
E-mail: {\tt kalock@asu.edu}\\ \\  

\noindent {\it Wai Yan Pong}\\  
California State University Dominguez Hills \\
1000 E Victoria Street, Carson, CA 90747\\
E-mail: {\tt wpong@csudh.edu}\\ \\

\noindent {\it Alexander Wittmond}\\
University of Missouri \\
Columbia, MO 65211 \\
E-mail: {\tt ajwcz7@mail.missouri.edu} }\\ \\   
\end{document}